\documentclass[10pt,reqno]{amsart}

\usepackage{amsmath,amsthm,amsfonts,amssymb,latexsym,mathrsfs,color,extarrows}
\usepackage{hyperref}

\newtheorem{theorem}{Theorem}
\newtheorem{corollary}[theorem]{Corollary}
\newtheorem{proposition}[theorem]{Proposition}

\newtheorem{remark}[theorem]{Remark}

\newtheorem{lemma}[theorem]{Lemma}
\newtheorem{definition}[theorem]{Definition}
\newtheorem{example}[theorem]{Example}

\newcommand{\sn}{{\rm sn\,}}
\newcommand{\cn}{{\rm cn\,}}
\newcommand{\dn}{{\rm dn\,}}

\newcommand{\des}{{\rm des\,}}
\newcommand{\mss}{\mathfrak{S}}

\newcommand{\msn}{\mathfrak{S}_n}

\newcommand{\ms}{\mathfrak{S}}

\newcommand{\lrf}[1]{\lfloor #1\rfloor}
\newcommand{\lrc}[1]{\lceil #1\rceil}

\newcommand{\Eulerian}[2]{\genfrac{<}{>}{0pt}{}{#1}{#2}}

\title[Dumont differential system]{Several variants of the Dumont differential system and permutation statistics}
\author[S.-M.~Ma]{Shi-Mei~Ma}
\address{School of Mathematics and Statistics,
        Northeastern University at Qinhuangdao,
         Hebei 066000, P. R. China}
\email{shimeimapapers@163.com (S.-M. Ma)}
\author[T.~Mansour]{Toufik Mansour}
\address{Department of Mathematics, University of Haifa, Haifa 3498838, Israel}
\email{toufik@math.haifa.ac.il (T. Mansour)}
\author[D.G.L. Wang]{David G.L. Wang}
\address{School of Mathematics and Statistics, Beijing Institute of Technology, 102488 Beijing, P. R. China}
\email{glw@bit.edu.cn  (D.G.L. Wang)}
\author[Y.-N. Yeh]{Yeong-Nan Yeh}
\address{Institute of Mathematics,
        Academia Sinica, Taipei 10617, Taiwan}
\email{mayeh@math.sinica.edu.tw (Y.-N. Yeh)}
\begin{document}

\maketitle

\begin{abstract}
The Dumont differential system on the Jacobi elliptic functions was introduced by Dumont (Math Comp, 1979, 33: 1293--1297) and was extensively studied by Dumont, Viennot, Flajolet and so on. In this paper, we first present a labeling scheme for the cycle structure of permutations. We then introduce two types of Jacobi-pairs of differential equations. We present a general method to derive the solutions of
these differential equations. As applications, we present some characterizations for several permutation statistics.
\end{abstract}
\section{Introduction}
The Jacobi elliptic functions occur naturally in geometry, analysis, number theory, algebra and combinatorics (see~\cite{Dumont81,Flajolet89,Flajolet09,Viennot80} for instance).
The three basic {\it Jacobi elliptic functions} $\sn(u,k),\cn(u,k),\dn(u,k)$ are respectively defined by
\begin{align*}
&u=\int_{0}^{\sn(u,k)}\frac{d t}{\sqrt{(1-t^2)(1-k^2t^2)}}, \\
&\cn(u,k)=\sqrt{1-\sn^2(u,k)}, \\
&\dn(u,k)=\sqrt{1-k^2\sn^2(u,k)},
\end{align*}
where the modulus is often confined to the normal case $0<k<1$.
These functions are generalizations of the trigonometric functions and hyperbolic functions satisfying
$$\sn(u,0)=\sin u, \cn(u,0)=\cos u, \dn(u,0)=1,$$ $$\sn(u,1)=\tanh u, \cn(u,1)=\dn(u,1)=\mathrm{sech} u.$$
The Taylor series expansions of these Jacobian elliptic functions are given as follows:
\begin{align*}
&\sn(u,k)=u-(1+k^2)\frac{u^3}{3!}+(1+14k^2+k^4)\frac{u^5}{5!}+\cdots, \\
&\cn(u,k)=1-\frac{u^2}{2!}+(1+4k^2)\frac{u^4}{4!}-(1+44k^2+16k^4)\frac{u^6}{6!}+\cdots, \\
&\dn(u,k)=1-k^2\frac{u^2}{2!}+k^2(4+k^2)\frac{u^4}{4!}-k^2(16+44k^2+k^4)\frac{u^6}{6!}+\cdots.
\end{align*}
Using formal methods, Abel~\cite{Abel1826} discovered the following differential system:
\begin{equation}\label{Abel-diff}
\left\{
  \begin{array}{ll}
    \frac{d}{du}\sn(u,k)=\cn(u,k)\dn(u,k),  \\
    \frac{d}{du}\cn(u,k)=-\sn(u,k)\dn(u,k),  \\
    \frac{d}{du}\dn(u,k)=-k^2\sn(u,k)\cn(u,k).
  \end{array}
\right.
\end{equation}

Let $\msn$ denote the symmetric group of all permutations of $[n]$, where $[n]=\{1,2,\ldots,n\}$.
An {\it interior peak} in $\pi$ is an index $i\in\{2,3,\ldots,n-1\}$ such that $\pi(i-1)<\pi(i)>\pi(i+1)$.
Given a permutation $\pi\in\msn$, a value $i\in[n]$ is called a {\it cycle peak} if $\pi^{-1}(i)<i>\pi(i)$.
Throughout this paper, we always let $X(\pi)$ (resp., $Y(\pi)$) be the number of odd (resp., even) cycle peaks of $\pi$. For example,
for $\pi=241365$, we have $X(\pi)=0$ and $Y(\pi)=2$.

Let $D$ be the derivative operator on the polynomials in three
variables. The {\it Dumont differential system} on the Jacobi elliptic functions is
defined by
\begin{equation}\label{diff-elliptic}
\left\{
  \begin{array}{ll}
    D(x)=yz,  \\
    D(y)=xz,\\
    D(z)=xy.
  \end{array}
\right.
\end{equation}
For $n\geq0$, we define the numbers $s_{n,i,j}$ by
\begin{equation}\label{Schett-Dumontsystem}
\begin{split}
D^{2n}(x)&=\sum_{i,j\geq 0}s_{2n,i,j}x^{2i+1}y^{2j}z^{2n-2i-2j},\\
D^{2n+1}(x)&=\sum_{i,j\geq 0}s_{2n+1,i,j}x^{2i}y^{2j+1}z^{2n-2i-2j+1}.
\end{split}
\end{equation}

The study of~\eqref{diff-elliptic} was initiated by Schett~\cite{Schett76} (in a slightly
different form) and he found that $$\sum_{i,j\geq0}s_{n,i,j}=n!,~\sum_{j\geq0}s_{n,i,j}=P_{n,\lrf{(n-1)/{2}}-i},$$
where $P_{n,k}$ is the number of permutations in $\msn$ with $k$ interior peaks.
Dumont~\cite{Dumont79} deduced the recurrence relation
\begin{equation}\label{Dumont-recu}
\begin{split}
s_{2n,i,j}&=(2j+1)s_{2n-1,i,j}+(2i+2)s_{2n-1,i+1,j-1}+(2n-2i-2j+1)s_{2n-1,i,j-1},\\
s_{2n+1,i,j}&=(2i+1))s_{2n,i,j}+(2j+2)s_{2n,i-1,j+1}+(2n-2i-2j+2)s_{2n,i-1,j},
\end{split}
\end{equation}
and established that
\begin{equation}\label{snij-Dumont}
s_{n,i,j}=|\{\pi\in\msn: X(\pi)=i, Y(\pi)=j\}|.
\end{equation}
Moreover, Dumont~\cite[Corollary~1]{Dumont79} obtained the following result:
\begin{itemize}
          \item [(i)] the coefficient of $(-1)^nk^{2j}{u^{2n+1}}/{(2n+1)!}$ in the
Taylor expansion of $\sn(u,k)$ is equal to the number of permutations in $\ms_{2n}$ (or in $\ms_{2n+1}$)
having $j$ even cycle peaks and with no odd cycle peaks;
          \item [(ii)] the coefficient of $(-1)^nk^{2i}{u^{2n}}/{(2n)!}$ (resp. $(-1)^nk^{2n-2i}{u^{2n}}/{(2n)!}$) in the
Taylor expansion of $\cn(u,k)$ (resp. $\dn(u,k)$) is equal to the number of permutations in $\ms_{2n-1}$ (or in $\ms_{2n}$) having $i$ odd cycle peaks and with no even cycle peaks.
        \end{itemize}
Subsequently,
Dumont~\cite{Dumont81} studied the
symmetric variant of~\eqref{Abel-diff}:
\begin{equation*}\label{Dumount-diff}
\begin{split}
\frac{d}{du}\sn(u;a,b)&=\cn(u;a,b)\dn(u;a,b),\\
\frac{d}{du}\cn(u;a,b)&=a^2\sn(u;a,b)\dn(u;a,b),\\
\frac{d}{du}\dn(u;a,b)&=b^2\sn(u;a,b)\cn(u;a,b),
\end{split}
\end{equation*}
with the initial conditions $\sn(0;a,b)=0, \cn(0;a,b)=1$ and $\dn(0;a,b)=1$.
In particular, for the Dumont differential system~\eqref{diff-elliptic},
Dumont~\cite[Proposition 2.1]{Dumont81} showed that
\begin{equation}\label{Dumont-explicit}
\sum_{n\geq 0}D^n(x)\frac{u^n}{n!}=\frac{yz\sn(u;v,w)+x\cn(u;v,w)\dn(u;v,w)}{1-x^2\sn^2(u;v,w)},
\end{equation}
where $v=\sqrt{y^2-x^2}$ and $w=\sqrt{z^2-x^2}$.

The grammatical method was systematically introduced by Chen~\cite{Chen93}
in the study of exponential structures in combinatorics. Many combinatorial structures can be generated by using context-free
grammars. We refer the reader to~\cite{Chen17,Ma1302,Ma1303} for recent progress on this topic.
Let $A$ be an alphabet whose letters are regarded as independent commutative indeterminates.
A {\it context-free grammar} $G$ over $A$ is defined as a set
of substitution rules that replace a letter in $A$ by a formal function over $A$.
The formal derivative $D$ is a linear operator defined with respect to a context-free grammar $G$.
It is clear that~\eqref{diff-elliptic} is equivalent to the context-free grammar
\begin{equation}\label{grammar-Schett}
G=\{x\rightarrow yz, y\rightarrow xz,z\rightarrow xy\}.
\end{equation}

This paper is organized as follows. In Section~\ref{Section-2}, we present
a constructive proof of~\eqref{snij-Dumont} by using the grammatical labeling introduced by Chen and Fu~\cite{Chen17}.
In Section~\ref{Section-3},
we introduce and study two types of Jacobi-pairs of differential equations.
In Section~\ref{Section-4}, we present some characterizations for several permutation statistics.
\section{A constructive proof of~\eqref{snij-Dumont}}\label{Section-2}
In this section, we always write $\pi\in\msn$ using the {\it standard
cycle decomposition}, where each cycle is written with its smallest entry first and the
cycles are written in increasing order of their smallest entry.
In what follows, we present a labeling scheme for the cycle structure of permutations.

Let $$\mss_{n,i,j}=\{\pi\in\msn: X(\pi)=i, Y(\pi)=j\}.$$
\begin{definition}\label{def01}
Let $\pi\in\mss_{n,i,j}$. 
Then we put the superscript label $x$ immediately before and right after each odd cycle peak of $\pi$, and
we put the superscript label $y$ immediately before and right after each even cycle peak.
In each of the remaining positions
except the first position of each cycle, we put the superscript label $z$. Moreover, we put the superscript label $x$ (resp. $y$) at the end of $\pi$ if $n$ is even (resp. odd).
\end{definition}

For example, for $\pi=(132)(45)(68)(7)\in\mss_{8,2,1}$ and $\pi'=(132)(45)\in\mss_{5,2,0}$, the labeled $\pi$ and $\pi'$ are respectively given by $$(1^{x}3^{x}2^z)(4^{x}5^{x})(6^{y}8^{y})(7^z)^x,~(1^{x}3^{x}2^z)(4^{x}5^{x})^y.$$

When $n=1$, we have $\mss_{1,0,0}=\{(1^z)^y\}$. When $n=2$, we have $\mss_{2,0,0}=\{(1^z)(2^z)^x\}$ and $\mss_{2,0,1}=\{(1^y2^y)^x\}$.
Let $n=m$. Suppose we get all labeled permutations in $\mss_{m,i,j}$ for all $i,j$, where $m\geq 2$. We now consider the case $n=m+1$.
Let $\widehat{\pi}\in\mss_{m+1}$ be obtained from $\pi\in\mss_{m,i,j}$ by inserting the entry $m+1$ into $\pi$.
In the following, we construct a correspondence, denoted by $\Phi$, between $\pi$ and $\widehat{\pi}$.

If $m$ is odd and the entry $m+1$ is inserted at the end of $\pi$ as a new cycle $(m+1)$, then
we leave all labels of $\pi$ unchanged except the last label $y$. We define $\Phi$ by
$$\pi=\cdots (\cdots)^y\leftrightarrow\widehat{\pi}=\cdots (\cdots)((m+1)^{z})^x,$$
which corresponds to the operation $y\rightarrow xz$. Note that $X(\widehat{\pi})=X(\pi)$ and $Y(\widehat{\pi})=Y(\pi)$.
Hence $\widehat{\pi}\in\mss_{m+1,i,j}$.
If $m$ is odd and the entry $m+1$ occurs in a cycle with at least two elements,
there are three cases to consider:
\begin{enumerate}
  \item [\rm ($i$)] Suppose $c_r$ is the $r$th odd cycle peak of $\pi$ and we put the entry $m+1$ immediately before or right after $c_r$. Then we have
 $$\pi=\cdots (\ldots^xc_r~^x\ldots)\cdots(\cdots)^y\leftrightarrow\widehat{\pi}=\cdots (\ldots ^y(m+1)^yc_r~^z\ldots)\cdots(\cdots)^x,$$
  or $$\pi=\cdots (\ldots^xc_r~^x\ldots)\cdots(\cdots)^y\leftrightarrow\widehat{\pi}=\cdots (\ldots ^zc_r~^y(m+1)^y\ldots)\cdots(\cdots)^x.$$
In this case, the corresponding operation of $\Phi$ is $x\rightarrow yz$ and we have $\widehat{\pi}\in\mss_{m+1,i-1,j+1}$.
 \item [\rm ($ii$)] Suppose $d_\ell$ is the $\ell$th even cycle peak of $\pi$ and we put the entry $m+1$ immediately before or right after $d_\ell$. Then we have
 $$\pi=\cdots (\ldots^yd_\ell~^y\ldots)\cdots(\cdots)^y\leftrightarrow\widehat{\pi}=\cdots (\ldots ^y(m+1)^yd_\ell~^z\ldots)\cdots(\cdots)^x,$$
  or $$\pi=\cdots (\ldots^yd_\ell~^y\ldots)\cdots(\cdots)^y\leftrightarrow\widehat{\pi}=\cdots (\ldots ^zd_{\ell}~^y(m+1)^y\ldots)\cdots(\cdots)^x.$$
In this case, the corresponding operation of $\Phi$ is $y\rightarrow xz$ and we have $\widehat{\pi}\in\mss_{m+1,i,j}$.
  \item [\rm ($iii$)] If we insert $m+1$ into a position of $\pi$ with label $z$, then we have
  $$\pi=\cdots (\ldots w~^{z}\ldots)\cdots(\cdots)^y\leftrightarrow\widehat{\pi}=\cdots (\ldots w ^y(m+1)^y\ldots)\cdots(\cdots)^x.$$
  In this case, the corresponding operation of $\Phi$ is $z\rightarrow xy$ and we have $\widehat{\pi}\in\mss_{m+1,i,j+1}$.
  \end{enumerate}
If $m$ is even and the entry $m+1$ is inserted at the end of $\pi$ as a new cycle $(m+1)$, then
we leave all labels of $\pi$ unchanged except the last label $x$. We define $\Phi$ by
$$\pi=\cdots (\cdots)^x\leftrightarrow\widehat{\pi}=\cdots (\cdots)((m+1)^{z})^y,$$
which corresponds to the operation $x\rightarrow yz$. In this case, we have $\widehat{\pi}\in\mss_{m+1,i,j}$.
If $m$ is even and the entry $m+1$ occurs in a cycle with at least two elements,
there are also three cases to consider:
\begin{enumerate}
  \item [\rm ($i$)] Suppose $c_r$ is the $r$th odd cycle peak of $\pi$ and we put the entry $m+1$ immediately before or right after $c_r$. Then we have
  $$\pi=\cdots (\ldots^xc_r~^x\ldots)\cdots(\cdots)^x\leftrightarrow\widehat{\pi}=\cdots (\ldots ^x(m+1)^xc_r~^z\ldots)\cdots(\cdots)^y,$$
  or $$\pi=\cdots (\ldots^xc_r~^x\ldots)\cdots(\cdots)^x\leftrightarrow\widehat{\pi}=\cdots (\ldots ^zc_r~^x(m+1)^x\ldots)\cdots(\cdots)^y.$$
 In this case, the corresponding operation of $\Phi$ is $x\rightarrow yz$ and we have $\widehat{\pi}\in\mss_{m+1,i,j}$.
 \item [\rm ($ii$)] Suppose $d_\ell$ is the $\ell$th even cycle peak of $\pi$ and we put the entry $m+1$ immediately before or right after $d_\ell$. Then we have
 $$\pi=\cdots (\ldots^yd_\ell~^y\ldots)\cdots(\cdots)^x\leftrightarrow\widehat{\pi}=\cdots (\ldots ^x(m+1)^xd_\ell~^z\ldots)\cdots(\cdots)^y,$$
  or $$\pi=\cdots (\ldots^yd_\ell~^y\ldots)\cdots(\cdots)^x\leftrightarrow\widehat{\pi}=\cdots (\ldots ^zd_\ell~^x(m+1)^x\ldots)\cdots(\cdots)^y.$$
 In this case, the corresponding operation of $\Phi$ is $y\rightarrow xz$ and we have $\widehat{\pi}\in\mss_{m+1,i+1,j-1}$.
  \item [\rm ($iii$)] If we insert $m+1$ into a position of $\pi$ with label $z$, then we have
  $$\pi=\cdots (\ldots w~^{z}\ldots)\cdots(\cdots)^x\leftrightarrow\widehat{\pi}=\cdots (\ldots w ^x(m+1)^x\ldots)\cdots(\cdots)^y.$$
 In this case, the corresponding operation of $\Phi$ is $z\rightarrow xy$ and we have $\widehat{\pi}\in\mss_{m+1,i+1,j}$.
  \end{enumerate}
By induction and~\eqref{Dumont-recu}, we see that $\Phi$ is the desired correspondence between permutations in $\mss_{m}$ and $\mss_{m+1}$,
which also gives a constructive proof of~\eqref{snij-Dumont}.

\begin{example}
Given $\pi=(14)(23)\in\mss_{4,1,1}$.
The correspondence between $\pi$ and $x^3y^2$ is built up as follows:
$$(1^z)^\textbf{y}\leftrightarrow{ y\rightarrow xz} (1^z)(2^\textbf{z})^x\leftrightarrow{z\rightarrow xy}  (1^\textbf{z})(2^x3^x)^y\leftrightarrow{z\rightarrow xy} (1^y4^y)(2^x3^x)^x.$$
\end{example}
\section{Solutions of two types of Jacobi-pairs}\label{Section-3}
\subsection{Basic definitions and notation}
\hspace*{\parindent}

Let
\begin{equation}\label{ellipticintegral}
F(x,k)=\int_{0}^{x}\frac{d t}{\sqrt{(1-t^2)(1-k^2t^2)}},
\end{equation}
which is the incomplete elliptic integral of the first kind in Jacobi's form.
Define
\begin{equation*}
\begin{split}
h_{p,q}&=F\left(\sqrt{\frac{q(1-p)}{q-p}},\sqrt{\frac{q-p}{1-p}}\right),\\
\ell_{p,q}&=F\left(q\sqrt{\frac{1-p}{q^2-p}},\sqrt{\frac{q^2-p}{1-p}}\right),\\
k_{p,q}&=\sqrt{\frac{p-1}{q-p}}\arctan\left(\sqrt{\frac{q(p-1)}{q-p}}\right),\\
x_\pm &=(p-1)x\pm k_{p,q}.
\end{split}
\end{equation*}

For any sequence $a_{n,i,j}$, we define the following generating functions
\begin{align*}
A&=A(x,p,q)=\sum_{n,i,j\geq0}a_{n,i,j}\frac{x^{n}}{n!}p^iq^j,\\
AE&=AE(x,p,q)=\sum_{n,i,j\geq0}a_{2n,i,j}\frac{x^{2n}}{(2n)!}p^iq^j=\frac12(A(x,p,q)+A(-x,p,q)),\\
AO&=AO(x,p,q)=\sum_{n,i,j\geq0}a_{2n+1,i,j}\frac{x^{2n+1}}{(2n+1)!}p^iq^j=\frac12(A(x,p,q)-A(-x,p,q)),
\end{align*}
where we use the small letters $a,b,c,\ldots$ for sequences, capital letters $A,B,C,\ldots$ for generating functions, and $AE,BE,CE,\ldots,AO,BO,CO,\ldots$ for the even and odd parts of the generating functions, respectively. Also, we denote by $H_y$ the partial derivative of the function $H$ with respect to $y$.

Recall that the numbers $s_{n,i,j}$ are defined by~\eqref{Schett-Dumontsystem}. Then
\begin{align*}
S=S(x,p,q)&=\sum_{n,i,j\geq0}s_{n,i,j}\frac{x^{n}}{n!}p^iq^j,\\
SE=SE(x,p,q)&=\sum_{n,i,j\geq0}s_{2n,i,j}\frac{x^{2n}}{(2n)!}p^iq^j\mbox{ and }\\
SO=SO(x,p,q)&=\sum_{n,i,j\geq0}s_{2n+1,i,j}\frac{x^{2n+1}}{(2n+1)!}p^iq^j.
\end{align*}
Using~\eqref{Dumont-recu}, we get the following comparable result of~\eqref{Dumont-explicit}.
\begin{theorem}\label{lem:aa0}
We have
\begin{align}\label{eqss0}
\left\{\begin{array}{ll}
SO(x,p,q)&=\frac{\sqrt{p-1}}{2\sqrt{q}}\left(K\left(\frac{1-q}{1-p},\sqrt{p-1}x-h_{p,q}\right)-K\left(\frac{1-q}{1-p},\sqrt{p-1}x+h_{p,q}\right)\right),\\
SE(x,p,q)&=\frac{\sqrt{p-1}}{2\sqrt{p}}\left(K\left(\frac{1-q}{1-p},\sqrt{p-1}x-h_{p,q}\right)+K\left(\frac{1-q}{1-p},\sqrt{p-1}x+h_{p,q}\right)\right),
\end{array}\right.
\end{align}
where $K(p,x)=\sqrt{1-p}\cn(\sqrt{p}x,\sqrt{1-1/p}),~-1<p<1$ and $0<q<1$.
\end{theorem}
\begin{proof}
By \eqref{Dumont-recu}, we have
\begin{align*}
\left\{\begin{array}{ll}
SO_x&=SE+2p(1-p)SE_p+2p(1-q)SE_q+pxSE_x,\\
SE_x&=SO+2q(1-p)SO_p+2q(1-q)SO_q+qxSO_x.
\end{array}\right.
\end{align*}
Set $$(\widetilde{SO},\widetilde{SE})=\frac{1}{\sqrt{p-1}}(\sqrt{q}SO,\sqrt{p}SE).$$ Then
\begin{align}\label{Dumont-J-pair}
\left\{\begin{array}{ll}
\widetilde{SO}_x&=2\sqrt{pq}(1-p)\widetilde{SE}_p+2\sqrt{pq}(1-q)\widetilde{SE}_q+\sqrt{pq}x\widetilde{SE}_x,\\
\widetilde{SE}_x&=2\sqrt{pq}(1-p)\widetilde{SO}_p+2\sqrt{pq}(1-q)\widetilde{SO}_q+\sqrt{pq}x\widetilde{SO}_x.
\end{array}\right.
\end{align}
Solving (\ref{Dumont-J-pair}) for $\widetilde{SO}_x-\widetilde{SE}_x$ and $\widetilde{SO}_x+\widetilde{SE}_x$ (with the help of maple), we obtain that there exist two (analytical) functions $K_1$ and $K_2$ such that
\begin{align}\label{eqss1}
\left\{\begin{array}{ll}
\widetilde{SO}-\widetilde{SE}&=K_1\left(\frac{1-q}{1-p},\sqrt{p-1}x+h_{p,q}\right),\\
\widetilde{SO}+\widetilde{SE}&=K_2\left(\frac{1-q}{1-p},\sqrt{p-1}x-h_{p,q}\right).
\end{array}\right.
\end{align}
In order to provide explicit formulas for the generating functions $\widetilde{SO}_x$ and $\widetilde{SE}_x$, we solve (\ref{Dumont-J-pair}) for $q=0$. In this case, we obtain
\begin{align*}
\left\{\begin{array}{ll}
SO_x(x,p,0)&=SE(x,p,0)+2p(1-p)SE_p(x,p,0)+2pSE_q(x,p,0)+pxSE_x(x,p,0),\\
SE_x(x,p,0)&=SO(x,p,0).
\end{array}\right.
\end{align*}
Note that our initial conditions are $SO(0,p,q)=0$, $SE(0,p,q)=1$, $$SO(x,0,0)=\frac{e^x-e^{-x}}{2},~SE(x,0,0)=\frac{e^x+e^{-x}}{2}.$$ Thus, it is obvious to see that the solution of this system of partial differential equations is given by
$$SO(x,p,0)=-I\dn(Ix, \sqrt{p})\sn(Ix, \sqrt{p})\mbox{ and }SE(x,p,0)=\cn(Ix,\sqrt{p}),$$
with $I^2=-1$. Therefore, solving~\eqref{eqss1} for $q=0$ gives
\begin{align*}
-\frac{\sqrt{p}}{\sqrt{p-1}}\cn(Ix,\sqrt{p})&=K_1\left(\frac{1}{1-p},\sqrt{p-1}x\right),\\
\frac{\sqrt{p}}{\sqrt{p-1}}\cn(Ix,\sqrt{p})&=K_2\left(\frac{1}{1-p},\sqrt{p-1}x\right),
\end{align*}
which leads to $K_2(p,x)=-K_1(p,x)=K(p,x)$. Hence, by (\ref{eqss1}) we get (\ref{eqss0}), as claimed.
\end{proof}

In order to provide a unified approach to the sequences discussed in this paper,
we introduce the following definitions.

\begin{definition}
A pair $(F,G)=(F(x,p,q),G(x,p,q))$ of functions is called {\it the Jacobi-pair of the first type} if they satisfy the following system of PDEs:
\begin{align}\label{eqJpair}
\left\{\begin{array}{ll}
F_x&=2p\sqrt{q}(1-p)G_p+2p\sqrt{q}(1-q)G_q+2p\sqrt{q}xG_x,\\
G_x&=2p\sqrt{q}(1-p)F_p+2p\sqrt{q}(1-q)F_q+2p\sqrt{q}xF_x.
\end{array}\right.
\end{align}
\end{definition}

\begin{remark}\label{GenJP}
Concerning the solution to~\eqref{eqJpair}, note that by defining
$$P(x,p,q)=F(x,p,q)-G(x,p,q), ~Q(x,p,q)=F(x,p,q)+G(x,p,q),$$ we have
\begin{align*}
\left\{\begin{array}{ll}
P_x(x,p,q)+2p\sqrt{q}((1-p)P_p(x,p,q)+(1-q)P_q(x,p,q)+xP_x(x,p,q))=0,\\
Q_x(x,p,q)-2p\sqrt{q}((1-p)Q_p(x,p,q)+(1-q)Q_q(x,p,q)+xQ_x(x,p,q))=0.
\end{array}\right.
\end{align*}
Using the Maple package,
it is not hard to check that the solution (with $p,q\neq1$ and $q\neq0$) of these PDEs is given by
\begin{equation*}
P(x,p,q)=V\left(\frac{1-q}{1-p},x_+\right),~
Q(x,p,q)=\widetilde{V}\left(\frac{1-q}{1-p},x_-\right),
\end{equation*}
for any two functions $V$ and $\widetilde{V}$.
\end{remark}

\begin{definition}
A pair $(M,N)=(M(x,p,q),N(x,p,q))$ of functions is called {\it the Jacobi-pair of the second type} if they satisfy the following system of PDEs:
\begin{align}\label{eqJpair-2}
\left\{\begin{array}{ll}
M_x&=2q\sqrt{p}(1-p)N_p+\sqrt{p}(1-q^2)N_q+xq\sqrt{p}N_x,\\
N_x&=2q\sqrt{p}(1-p)M_p+\sqrt{p}(1-q^2)M_q+xq\sqrt{p}M_x.
\end{array}\right.
\end{align}
\end{definition}

\begin{remark}\label{GenJP-2}
Concerning the solution to~\eqref{eqJpair-2}, note that by defining
$$\widetilde{P}(x,p,q)=M(x,p,q)-N(x,p,q), ~\widetilde{Q}(x,p,q)=M(x,p,q)+N(x,p,q),$$ we have
\begin{align*}
\left\{\begin{array}{ll}
\widetilde{P}_x(x,p,q)+\sqrt{q}(2q(1-p)\widetilde{P}_p(x,p,q)+(1-q^2)\widetilde{P}_q(x,p,q)+xq\widetilde{P}_x(x,p,q))=0,\\
\widetilde{Q}_x(x,p,q)-\sqrt{q}(2q(1-p)\widetilde{Q}_p(x,p,q)+(1-q^2)\widetilde{Q}_q(x,p,q)+xq\widetilde{Q}_x(x,p,q))=0.
\end{array}\right.
\end{align*}
Using the Maple package,
it is not hard to check that the solution (with $p,q\neq1$ and $q\neq0$) of these PDEs is given by
\begin{equation*}
\widetilde{P}(x,p,q)=W\left(\frac{1-q^2}{1-p},\sqrt{p-1}x-\ell_{p,q}\right),~
\widetilde{Q}(x,p,q)=\widetilde{W}\left(\frac{1-q^2}{1-p},\sqrt{p-1}x+\ell_{p,q}\right),
\end{equation*}
for any two functions $W$ and $\widetilde{W}$.
\end{remark}
\subsection{Jacobi-pairs of the first type}
\hspace*{\parindent}

There are countless combinatorial structures related to the differential operators $xD$ and $Dx$ (e.g.,~\cite{Flajolet09,Joyal81,Ma13}). It is natural to further study~\eqref{diff-elliptic} via these differential operators.

Write
$$(xD)^{n+1}(x)=(xD)(xD)^n(x)=xD((xD)^n(x)),$$
$$(Dx)^{n+1}(x)=(Dx)(Dx)^n(x)=D(x(Dx)^n(x)),$$
$$(Dx)^{n+1}(y)=(Dx)(Dx)^n(y)=D(x(Dx)^n(y)).$$

In particular, from~\eqref{diff-elliptic}, we have
\begin{equation*}
\begin{split}
(xD)(x)&=xyz,\quad
(xD)^2(x)=xy^2z^2+x^3y^2+x^3z^2,\\
(Dx)(x)&=2xyz,\quad
(Dx)^2(x)=4xy^2z^2+2x^3y^2+2x^3z^2,\\
(Dx)(y)&=y^2z+x^2z,\quad
(Dx)^2(y)=y^3z^2+5x^2yz^2+x^2y^3+x^4y.
\end{split}
\end{equation*}

For $n\geq 0$, we define the numbers $a_{n,i,j},c_{n,i,j}$ and $d_{n,i,j}$ by
\begin{align*}
(xD)^{2n}(x)&=\sum_{i,j\geq 0}a_{2n,i,j}x^{2i+1}y^{2j}z^{4n-2i-2j},\\
(xD)^{2n+1}(x)&=\sum_{i,j\geq 0}a_{2n+1,i,j}x^{2i+1}y^{2j+1}z^{4n-2i-2j+1},\\
(Dx)^{2n}(x)&=\sum_{i,j\geq 0}c_{2n,i,j}x^{2i+1}y^{2j}z^{4n-2i-2j},\\
(Dx)^{2n+1}(x)&=\sum_{i,j\geq 0}c_{2n+1,i,j}x^{2i+1}y^{2j+1}z^{4n-2i-2j+1},\\
(Dx)^{2n}(y)&=\sum_{i,j\geq 0}d_{2n,i,j}x^{2i}y^{2j+1}z^{4n-2i-2j},\\
(Dx)^{2n+1}(y)&=\sum_{i,j\geq 0}d_{2n+1,i,j}x^{2i}y^{2j}z^{4n-2i-2j+3}.
\end{align*}

For convenience, we list the first terms of the corresponding generating functions:
\begin{align*}
A(x,p,q)&=1+x+(p(1+q)+q)\frac{x^2}{2!}+(4p^2+5p(1+q)+q)\frac{x^3}{3!}\\
&+(p^3(4+4q)+p^2(5+50q+5q^2)+p(18q^2+18q)+q^2)\frac{x^4}{4!}\\
&+(16p^4+p^3(148+148q)+p^2(61+394q+61q^2)+p(58q+58q^2)+q^2)\frac{x^5}{5!}+\cdots,
\end{align*}
\begin{align*}
C(x,p,q)&=1+2x+2(p(1+q)+2q)\frac{x^2}{2!}+8(p^2+2p(1+q)+q)\frac{x^3}{3!}\\
&+8(p^3(1+q)+2p^2(1+9q+q^2)+11pq(1+q)+2q^2)\frac{x^4}{4!}\\
&+16(2p^4+26p^3(1+q)+p^2(17+98q+17q^2)+26pq(1+q)+2q^2)\frac{x^5}{5!}+\cdots,
\end{align*}
\begin{align*}
D(x,p,q)&=1+(p+q)x+(p^2+p(5+q)+q)\frac{x^2}{2!}+(p^3+p^2(5+18q)+pq(18+5q)+q^2)\frac{x^3}{3!}\\
&+(p^4+p^3(58+18q)+p^2(61+164q+5q^2)+pq(58+18q)+q^2)\frac{x^4}{4!}+\cdots.
\end{align*}

Note that
\begin{align*}
(xD)^{2n+1}(x)&=(xD)(xD)^{2n}(x)\\
&=xD\left(\sum_{i,j\geq 0}a_{2n,i,j}x^{2i+1}y^{2j}z^{4n-2i-2j}\right)\\
&=\sum_{i,j\geq 0}(2i+1)a_{2n,i,j}x^{2i+1}y^{2j+1}z^{4n-2i-2j+1}+\sum_{i,j\geq 0}2ja_{2n,i,j}x^{2i+3}y^{2j-1}z^{4n-2i-2j+1}+\\
&\quad \sum_{i,j\geq 0}(4n-2i-2j)a_{2n,i,j}x^{2i+3}y^{2j+1}z^{4n-2i-2j-1}.
\end{align*}
Hence
\begin{equation}\label{a-recu-1}
a_{2n+1,i,j}=(2i+1)a_{2n,i,j}+(2j+2)a_{2n,i-1,j+1}+(4n-2i-2j+2)a_{2n,i-1,j}.
\end{equation}
Similarly,
\begin{equation}\label{a-recu-2}
a_{2n,i,j}=(2i+1)a_{2n-1,i,j-1}+(2j+1)a_{2n-1,i-1,j}+(4n-2i-2j+1)a_{2n-1,i-1,j-1}.
\end{equation}

Equivalently, recurrences~\eqref{a-recu-1} and~\eqref{a-recu-2} can be written as the following lemma.
\begin{lemma}\label{lem:aa1}
We have
\begin{align*}
\left\{\begin{array}{ll}
AO_x&=AE+2p(1-p)AE_p+2p(1-q)AE_q+2xpAE_x,\\
AE_x&=(q+p-pq)AO+2pq(1-p)AO_p+2pq(1-q)AO_q+2xpqAO_x.
\end{array}\right.
\end{align*}
Equivalently, $(\widetilde{AO},\widetilde{AE})$ is a Jacobi-pair of the first type, where $\widetilde{AO}=\sqrt{\frac{pq}{p-1}}AO$ and $\widetilde{AE}=\sqrt{\frac{p}{p-1}}AE$.
\end{lemma}

\begin{theorem}\label{th-aa}
Let $y=\frac{1-q}{1-p}$. Define
$$G(x,p)=\sqrt{\frac{1-p}{\cos^2(x\sqrt{p(1-p)})-p}}\mbox{ and }H(x,p)=\frac{(1-p)\sin(2x\sqrt{p(1-p)})}{2\sqrt{p}(\cos^2(x\sqrt{p(1-p)})-p)^{3/2}}.$$
Then
\begin{align*}
AO(x,p,q)=\frac{1}{2}\sqrt{\frac{p-q}{pq}}(H(yx_-,1-1/y)-G(yx_+,1-1/y)),\\
AE(x,p,q)=\frac{1}{2}\sqrt{\frac{p-q}{p}}(H(yx_-,1-1/y)+G(yx_+,1-1/y)).
\end{align*}
\end{theorem}
\begin{proof}
By Remark \ref{GenJP} and Lemma \ref{lem:aa1}, we obtain that
$$\sqrt{\frac{pq}{p-1}}AO(x,p,q)-\sqrt{\frac{p}{p-1}}AE(x,p,q)=V(y,x_+)$$
and
$$\sqrt{\frac{pq}{p-1}}AO(x,p,q)+\sqrt{\frac{p}{p-1}}AE(x,p,q)=\widetilde{V}(y,x_-),$$
for some functions $V$ and $\widetilde{V}$. Moreover, at $q=0$, the above equations reduce to
$$-V(p,x)=\widetilde{V}(p,x)=\sqrt{1-p}AE(-px,1-1/p,0).$$

Hence, if we guess that $AE(x,p,0)=G(x,p)$ and $AO(x,p,0)=H(x,p)$, then we get
\begin{align*}
\sqrt{\frac{pq}{p-1}}AO(x,p,q)-\sqrt{\frac{p}{p-1}}AE(x,p,q)=-\sqrt{1-y}G(yx_+,1-1/y),\\
\sqrt{\frac{pq}{p-1}}AO(x,p,q)+\sqrt{\frac{p}{p-1}}AE(x,p,q)=\sqrt{1-y}H(yx_-,1-1/y),
\end{align*}
which implies
\begin{align*}
AO(x,p,q)=\frac{1}{2}\sqrt{\frac{(1-y)(p-1)}{pq}}(H(yx_-,1-1/y)-G(yx_+,1-1/y)),\\
AE(x,p,q)=\frac{1}{2}\sqrt{\frac{(p-1)(1-y)}{p}}(H(yx_-,1-1/y)+G(yx_+,1-1/y)).
\end{align*}
To complete the proof, we have to check that the functions $AO$ and $AE$ are satisfying Lemma~\ref{lem:aa1}, which is a
routine procedure.
\end{proof}

Along the same lines, we get
\begin{equation}\label{cn-recu}
\begin{split}
c_{2n,i,j}&=(2i+2)c_{2n-1,i,j-1}+(2j+1)c_{2n-1,i-1,j}+(4n-2i-2j+1)c_{2n-1,i-1,j-1},\\
c_{2n+1,i,j}&=(2i+2)c_{2n,i,j}+(2j+2)c_{2n,i-1,j+1}+(4n-2i-2j+2)c_{2n,i-1,j},
\end{split}
\end{equation}
which leads to the following result.

\begin{lemma}\label{lem:cc1} We have
\begin{align}\label{c-PDE}
\left\{\begin{array}{ll}
CO_x&=2CE+2p(1-p)CE_p+2p(1-q)CE_q+2xpCE_x,\\
CE_x&=(p+2q-pq)CO+2pq(1-p)CO_p+2pq(1-q)CO_q+2xpqCO_x.
\end{array}\right.
\end{align}
Equivalently, $(\widetilde{CO},\widetilde{CE})$ is a Jacobi-pair of the first type, where $\widetilde{CO}=\frac{p\sqrt{q}}{p-1}CO$ and $\widetilde{CE}=\frac{p}{p-1}CE$.
\end{lemma}

\begin{theorem}\label{thCCth}
Define $y=\frac{1-q}{1-p}$ and
$G(x,p)=\frac{1-p}{p\cos^2(x\sqrt{p-1})+1-p}$. Then
\begin{align}\label{eq:solcc}
\begin{array}{ll}
CO(x,p,q)&=\frac{p-1}{2p\sqrt{q}}(G(x_-,y)-G(x_+,y)),\\
CE(x,p,q)&=\frac{p-1}{2p}(G(x_-,y)+G(x_+,y)),\\
C(x,p,q)&=\frac{p-1}{2p\sqrt{q}}(G(x_-,y)-G(x_+,y))+\frac{p-1}{2p}(G(x_-,y)+G(x_+,y)).
\end{array}
\end{align}
\end{theorem}
\begin{proof}
By Remark \ref{GenJP} and Lemma \ref{lem:cc1}, we obtain that
$$\frac{p\sqrt{q}}{p-1}CO(x,p,q)-\frac{p}{p-1}CE(x,p,q)=\widetilde{V}(y,x_+)$$
and
$$\frac{p\sqrt{q}}{p-1}CO(x,p,q)+\frac{p}{p-1}CE(x,p,q)=V(y,x_-)$$
for some functions $V$ and $\widetilde{V}$. Moreover, at $q=0$, then above equations reduce to
$$V(1/(1-p),(p-1)x)=-\widetilde{V}(1/(1-p),(p-1)x).$$

Hence, if we take $(1-p)CE(-px,1-1/p,0)=G(x,p)$, $CO(x,p,q)=\frac{p-1}{2p\sqrt{q}}(G(x_-,y)-G(x_+,y))$ and $CE(x,p,q)=\frac{p-1}{2p}(G(x_-,y)+G(x_+,y))$, then \eqref{eq:solcc} is a solution for~\eqref{c-PDE}, where
$$V(1/(1-p),(p-1)x)=-\widetilde{V}(1/(1-p),(p-1)x)=(1-p)CE(-px,1-1/p,0)=G(x,p).$$
To complete the proof, we have to check that the functions $CO$ and $CE$ are satisfying Lemma~\ref{lem:cc1}, which is a
routine procedure.
\end{proof}

\begin{corollary}\label{CO:caseC}
We have
\begin{align*}
C(x,0,q)&=\cosh(2\sqrt{q}x)+\frac{1}{\sqrt{q}}\sinh(2\sqrt{q}x),\\
C(x,1,q)&=\frac{(x^2(q-1)+2x+1)}{(x^2(1-q)-2x+1)(x^2(1-q)+2x+1)},\\
C(x,p,0)&=\frac{(1-p)\sqrt{1-p}\sin(2x\sqrt{p(1-p)})}{\sqrt{p}(\cos^2(x\sqrt{p(1-p)})-p)^2}+\frac{1-p}{\cos^2(x\sqrt{p(1-p)})-p},\\
C(x,p,1)&=\frac{p-1}{p-e^{2x(p-1)}}.
\end{align*}
\end{corollary}
\begin{proof}
By applying Theorem \ref{thCCth} for $q=0$ or $p=1$, we obtain the formulas of $C(x,p,0)$ and $C(x,1,q)$.
Solving \eqref{c-PDE} for $p=0$, we obtain
\begin{align*}
CE(x,0,q)&=\alpha_q e^{2\sqrt{q}x}+\beta_q e^{-2\sqrt{q}x},\\
CO(x,0,q)&=\frac{1}{\sqrt{q}}(\alpha_q e^{2\sqrt{q}x}-\beta_q e^{-2\sqrt{q}x}).
\end{align*}
By using the initial conditions $CE(0,p,q)=1$ and $CO(0,p,q)=0$, we obtain $CE(x,0,p)=\cosh(2\sqrt{q}x)$ and $CO(x,0,q)=\frac{1}{\sqrt{q}}\sinh(2\sqrt{q}x)$, which completes the first part of the proof.

Again, solving \eqref{c-PDE} with $q=1$ for $CO(x,p,1)-CE(x,p,1)$ and $CO(x,p,1)+CE(x,p,1)$, we obtain
\begin{align*}
CO(x,p,1)-CE(x,p,1)&=\frac{p-1}{p}V(x(p-1)+\frac{1}{2}\ln p),\\
CO(x,p,1)+CE(x,p,1)&=\frac{p-1}{p}\widetilde{V}(x(p-1)-\frac{1}{2}\ln p),
\end{align*}
where $V,\widetilde{V}$ are two fixed functions. By the initial values $CE(0,p,q)=1$ and $CO(0,p,q)=0$, we get
$$V(y)=\frac{e^{2y}}{1-e^{2y}}\mbox{ and }\widetilde{V}(y)=\frac{1}{1-e^{2y}}.$$
Hence,
\begin{align*}
CO(x,p,1)-CE(x,p,1)&=\frac{(p-1)e^{2x(p-1)}}{1-pe^{2x(p-1)}},\\
CO(x,p,1)+CE(x,p,1)&=\frac{p-1}{p-e^{2x(p-1)}},
\end{align*}
which completes the proof.
\end{proof}

Along the same lines, we get
\begin{equation}\label{dn-recu}
\begin{split}
d_{2n,i,j}&=(2i+1)d_{2n-1,i,j}+(2j+2)d_{2n-1,i-1,j+1}+(4n-2i-2j+1)d_{2n-1,i-1,j},\\
d_{2n+1,i,j}&=(2i+1)d_{2n,i,j-1}+(2j+1)d_{2n,i-1,j}+(4n-2i-2j+4)d_{2n,i-1,j-1},
\end{split}
\end{equation}
which leads to the following result.

\begin{lemma}\label{lem:dd1}
We have
\begin{align}\label{d-PDE}
\left\{\begin{array}{ll}
DO_x&=(p+q)DE+2pq(1-p)DE_p+2pq(1-q)DE_q+2pqxDE_x,\\
DE_x&=(1+p)DO+2p(1-p)DO_p+2p(1-q)DO_q+2pxDO_x.
\end{array}\right.
\end{align}
Equivalently, $(\widetilde{DO},\widetilde{DE})$ is a Jacobi-pair of the first type, where $\widetilde{DO}=\sqrt{\frac{p}{p-1}}DO$ and $\widetilde{DE}=\sqrt{\frac{pq}{p-1}}DE$.
\end{lemma}

By similar arguments as in the proof of Theorem \ref{thCCth} with help from Remark \ref{GenJP} and Lemma \ref{lem:dd1}, we obtain the following result.
\begin{theorem}\label{thDDth}
Define $y=\frac{1-q}{1-p}$ and
$G(x,p)=\frac{\sinh(x\sqrt{p-1})}{1-\frac{p}{p-1}\cosh^2(x\sqrt{p-1})}$. Then
\begin{align*}
DO(x,p,q)&=\frac{\sqrt{p-1}}{2\sqrt{p}}(G(x_-,y)+G(x_+,y)),\\
DE(x,p,q)&=\frac{\sqrt{p-1}}{2\sqrt{pq}}(G(x_-,y)-G(x_+,y)),\\
D(x,p,q)&=\frac{\sqrt{p-1}}{2\sqrt{pq}}(G(x_-,y)-G(x_+,y))+\frac{\sqrt{p-1}}{2\sqrt{p}}(G(x_-,y)+G(x_+,y)).
\end{align*}
\end{theorem}

\begin{corollary}\label{CO:caseD}
Let $\widetilde{p}=\sqrt{p(p-1)}$. Then we have
\begin{align*}
D(x,p,0)&=\frac{(p-1)\cosh(x\widetilde{p})(\cosh^2(x\widetilde{p})-2+p)}{((p-1)\cosh^2(x\widetilde{p})-p\sinh^2(x\widetilde{p}))^2}
+\frac{\widetilde{p}\sinh(x\widetilde{p})}{p-\cosh^2(x\widetilde{p})},\\
D(x,1,q)&=\frac{(x^2(q-1)+2x-1)(x^2(1-q)+2x+1)(x^3(q-1)^2+x^2(q-1)-x(q+1)-1)}{(x^2(q-1)-2x\sqrt{q}+1)^2(x^2(q-1)+2x\sqrt{q}+1)^2},\\
D(x,p,1)&=\frac{(1 - p)e^{(1 - p)x}}{1-pe^{2(1 - p)x}}.
\end{align*}
\end{corollary}

From Corollary~\ref{CO:caseC} and Corollary~\ref{CO:caseD}, it is easy to verify that
\begin{align*}
C(x,1,q)&=\sum_{n\geq0}\sum_{k\geq 0}\binom{2n+1}{2k}q^kx^{2n}+\sum_{n\geq1}\sum_{k\geq 0}\binom{2n}{2k+1}q^kx^{2n-1},\\
D(x,1,q)&=\sum_{n\geq0}\sum_{k\geq 0}\binom{2n+1}{2k+1}q^kx^{2n}+\sum_{n\geq1}\sum_{k\geq 0}\binom{2n}{2k}q^kx^{2n-1}.
\end{align*}
\subsection{Jacobi-pairs of the second type}
\hspace*{\parindent}

In~\cite{Dumont96}, Dumont considered chains of general substitution rules on words. In particular, Dumont discovered the following.
\begin{proposition}
If
\begin{equation}\label{grammar-Dumont}
G=\{w\rightarrow wx, x\rightarrow wx\},
\end{equation}
then
\begin{equation*}
D^n(w)=\sum_{k=0}^{n-1}\Eulerian{n}{k}w^{k+1}x^{n-k},
\end{equation*}
where $\Eulerian{n}{k}$ is the {\it Eulerian number}, i.e., the number of permutations in $\msn$ with $k$ descents.
\end{proposition}
As a conjunction of~\eqref{grammar-Schett} and~\eqref{grammar-Dumont}, it is natural to consider the context-free grammar
\begin{equation}\label{Extension-grammar-Schett}
G=\{w\rightarrow wx, x\rightarrow yz, y\rightarrow xz,z\rightarrow xy\}.
\end{equation}

From~\eqref{Extension-grammar-Schett}, we have
\begin{equation*}
\begin{split}
D(w)&=wx,~D^2(w)=w(x^2+yz),~D^3(x)=w(x^3+xz^2+3xyz+xy^2),\\
D^4(w)&=w(x^4+10x^2yz+4x^2z^2+4x^2y^2+3y^2z^2+y^3z+yz^3),\\
D(w^2)&=2w^2x,~D^2(w^2)=w^2(4x^2+2yz),~D^3(w^2)=w^2(8x^3+12xyz+2xz^2+2xy^2).
\end{split}
\end{equation*}

For $n\geq 0$, we define the numbers $t_{n,i,j}$ and $r_{n,i,j}$ by
\begin{align*}
D^{2n}(w)&=w\sum_{i,j\geq 0}t_{2n,i,j}x^{2i}y^{j}z^{2n-2i-j},\\
D^{2n+1}(w)&=w\sum_{i,j\geq 0}t_{2n+1,i,j}x^{2i+1}y^{j}z^{2n-2i-j},\\
D^{2n}(w^2)&=w^2\sum_{i,j\geq 0}r_{2n,i,j}x^{2i}y^{j}z^{2n-2i-j},\\
D^{2n+1}(w^2)&=w^2\sum_{i,j\geq 0}r_{2n+1,i,j}x^{2i+1}y^{j}z^{2n-2i-j}.
\end{align*}

The first terms of the corresponding generating functions are given as follows:
\begin{align*}
T(x,p,q)&=1+x+(p+q)\frac{x^2}{2!}+(1+p+3q+q^2)\frac{x^3}{3!}\\
&+(p^2+4p+(10p+1)q+(4p+3)q^2+q^3)\frac{x^4}{4!}\\
&+(p^2+14p+1+(30p+15)q+(14p+29)q^2+15q^3+q^4)\frac{x^5}{5!}+\cdots,
\end{align*}
\begin{align*}
R(x,p,q)&=1+2x+(4p+2q)\frac{x^2}{2!}+(2+8p+12q+2q^2)\frac{x^3}{3!}\\
&+(16p+16p^2+(2+56p)q+(12+16p)q^2+2q^3)\frac{x^4}{4!}\\
&+(2+88p+32p^2+(60+240p)q+(148+88p)q^2+60q^3+2q^4)\frac{x^5}{5!}+\cdots.
\end{align*}

Note that
\begin{align*}
D^{2n+1}(w)&=D(D^{2n}(w))\\
&=D\left(w\sum_{i,j\geq 0}t_{2n,i,j}x^{2i}y^{j}z^{2n-2i-j}\right)\\
&=w\sum_{i,j\geq 0}t_{2n,i,j}x^{2i+1}y^{j}z^{2n-2i-j}+w\sum_{i,j\geq 0}2it_{2n,i,j}x^{2i-1}y^{j+1}z^{2n-2i-j+1}+\\
&w\sum_{i,j\geq 0}jt_{2n,i,j}x^{2i+1}y^{j-1}z^{2n-2i-j+1}+w\sum_{i,j\geq 0}(2n-2i-j)t_{2n,i,j}x^{2i+1}y^{j+1}z^{2n-2i-j-1}.
\end{align*}
Hence
\begin{equation}\label{t-recu-1}
t_{2n+1,i,j}=t_{2n,i,j}+(2i+2)t_{2n,i+1,j-1}+(j+1)t_{2n,i,j+1}+(2n-2i-j+1)t_{2n,i,j-1}.
\end{equation}
Similarly,
\begin{equation}\label{t-recu-2}
t_{2n,i,j}=t_{2n-1,i-1,j}+(2i+1)t_{2n-1,i,j-1}+(j+1)t_{2n-1,i-1,j+1}+(2n-2i-j+1)t_{2n-1,i-1,j-1}.
\end{equation}
By rewriting these recurrence relations in terms of generating functions $TE$ and $TO$, we obtain the following result.

\begin{lemma}\label{lem:tt1} We have
\begin{align}\label{t-PDE}
\left\{\begin{array}{ll}
TO_x&=TE+2q(1-p)TE_p+(1-q^2)TE_q+xqTE_x,\\
TE_x&=(p+q-qp)TO+2pq(1-p)TO_p+p(1-q^2)TO_q+xqpTO_x.
\end{array}\right.
\end{align}
Equivalently, $(\widetilde{TO},\widetilde{TE})$ is a Jacobi-pair of the second type, where $\widetilde{TO}=\sqrt{\frac{p(1+q)}{1-q}}TO$ and $\widetilde{TE}=\sqrt{\frac{1+q}{1-q}}TE$.
\end{lemma}

\begin{theorem}\label{th_TT}
Let
$\ell'_{p,q}=\sqrt{\frac{1-q^2}{1-p}}\ell_{p,q}$.
Then we have
\begin{align*}
\left\{\begin{array}{ll}
TO(x,p,q)&=\frac{q-1}{\sqrt{p(p-1)}}\sn(-\sqrt{q^2-1}x+\ell'_{p,q},\sqrt{\frac{p-q^2}{1-q^2}}),\\
TE(x,p,q)&=\sqrt{\frac{1-q}{1+q}}\dn(-\sqrt{q^2-1}x-\ell'_{p,q},\sqrt{\frac{p-q^2}{1-q^2}}).
\end{array}\right.
\end{align*}
\end{theorem}
\begin{proof}
By Remark~\ref{GenJP-2}, we see that Lemma \ref{lem:tt1} leads to
\begin{align}\label{eqtt11}
\left\{\begin{array}{ll}
\sqrt{\frac{p(1+q)}{1-q}}TO(x,p,q)-\sqrt{\frac{1+q}{1-q}}TE(x,p,q)&=W\left(\frac{1-q^2}{1-p},\sqrt{p-1}x-\ell_{p,q}\right),\\
\sqrt{\frac{p(1+q)}{1-q}}TO(x,p,q)+\sqrt{\frac{1+q}{1-q}}TE(x,p,q)&=\widetilde{W}\left(\frac{1-q^2}{1-p},\sqrt{p-1}x+\ell_{p,q}\right),\\
\end{array}\right.
\end{align}
for some functions $W$ and $\widetilde{W}$. Thus, at $q=0$, we have
\begin{align*}
\left\{\begin{array}{ll}
\sqrt{p}TO(I\sqrt{p}x,1-1/p,0)-TE(I\sqrt{p}x,1-1/p,0)&=W\left(p,x\right),\\
\sqrt{p}TO(I\sqrt{p}x,1-1/p,0)+TE(I\sqrt{p}x,1-1/p,0)&=\widetilde{W}\left(p,x\right),\\
\end{array}\right.
\end{align*}
where $I^2=-1$. Therefore, if we set
$$TE(x,p,0)=\dn(Ix,\sqrt{p})\mbox{ and }TO(x,p,0)=-I\sn(Ix,\sqrt{p}),$$
then
\begin{align*}
\left\{\begin{array}{ll}
-I\sqrt{p}\sn(-\sqrt{p}x,\sqrt{1-1/p})-\dn(-\sqrt{p}x,\sqrt{1-1/p})&=W\left(p,x\right),\\
-I\sqrt{p}\sn(-\sqrt{p}x,\sqrt{1-1/p})+\dn(-\sqrt{p}x,\sqrt{1-1/p})&=\widetilde{W}\left(p,x\right).
\end{array}\right.
\end{align*}
By~\eqref{eqtt11}, we obtain
\begin{align*}
\left\{\begin{array}{ll}
&\sqrt{\frac{p(1+q)}{1-q}}TO(x,p,q)-\sqrt{\frac{1+q}{1-q}}TE(x,p,q)\\
&\qquad=-\sqrt{\frac{q^2-1}{1-p}}\sn(-\sqrt{q^2-1}x+\ell'_{p,q},\sqrt{\frac{p-q^2}{1-q^2}})-\dn(-\sqrt{q^2-1}x+\ell'_{p,q},\sqrt{\frac{p-q^2}{1-q^2}}),\\ &\sqrt{\frac{p(1+q)}{1-q}}TO(x,p,q)+\sqrt{\frac{1+q}{1-q}}TE(x,p,q)\\
&\qquad=-\sqrt{\frac{q^2-1}{1-p}}\sn(-\sqrt{q^2-1}x-\ell'_{p,q},\sqrt{\frac{p-q^2}{1-q^2}})+\dn(-\sqrt{q^2-1}x-\ell'_{p,q},\sqrt{\frac{p-q^2}{1-q^2}}),
\end{array}\right.
\end{align*}
which implies
\begin{align*}
\left\{\begin{array}{ll}
TO(x,p,q)&=\frac{q-1}{\sqrt{p(p-1)}}\sn(-\sqrt{q^2-1}x+\ell'_{p,q},\sqrt{\frac{p-q^2}{1-q^2}}),\\
TE(x,p,q)&=\sqrt{\frac{1-q}{1+q}}\dn(-\sqrt{q^2-1}x-\ell'_{p,q},\sqrt{\frac{p-q^2}{1-q^2}}),
\end{array}\right.
\end{align*}
which agrees with the case $q=0$. To complete the proof, we have to check that the functions $TO$ and $TE$ satisfy Lemma~\ref{lem:tt1}, which is a  routine procedure.
\end{proof}

By the above theorem (or by a direct check using Lemma~\ref{lem:tt1}), we obtain the following result.

\begin{corollary}
Let $h(x,p)=\frac{\sqrt{p-1}}{\sqrt{p-1}\cosh(x\sqrt{p-1})-\sqrt{p}\sinh(x\sqrt{p-1})}$. Then, we have
\begin{align*}
T(x,p,1)&=\frac12(h(x,p)+h(-x,p))+\frac1{2\sqrt{p}}(h(x,p)-h(-x,p)),\\
T(x,1,q)&=\frac{q^2-1+\sqrt{q^2-1}\sinh(x\sqrt{q^2-1})}{(1+q)(q-\cosh(x\sqrt{q^2-1}))}.
\end{align*}
\end{corollary}

Along the same lines, we have
\begin{equation}\label{rn-recu}
\begin{split}
r_{2n+1,i,j}&=2r_{2n,i,j}+(2i+2)r_{2n,i+1,j-1}+(j+1)r_{2n,i,j+1}+(2n-2i-j+1)r_{2n,i,j-1},\\
r_{2n,i,j}&=2r_{2n-1,i-1,j}+(2i+1)r_{2n-1,i,j-1}+(j+1)r_{2n-1,i-1,j+1}+\\
& (2n-2i-j+1)r_{2n-1,i-1,j-1},
\end{split}
\end{equation}
which implies the following result.

\begin{lemma}\label{lem:rr1} We have
\begin{align}\label{r-PDE}
\left\{\begin{array}{ll}
RO_x&=2RE+2q(1-p)RE_p+(1-q^2)RE_q+xqRE_x,\\
RE_x&=(2p+q-pq)RO+2pq(1-p)RO_p+p(1-q^2)RO_q+xpqRO_x.
\end{array}\right.
\end{align}
Equivalently, $(\widetilde{RO},\widetilde{RE})$ is a Jacobi-pair of the second type, where $\widetilde{RO}=\frac{\sqrt{p}(1+q)}{1-q}RO$ and $\widetilde{RE}=\frac{1+q}{1-q}RE$.
\end{lemma}

Along the line of the proof of Theorem~\ref{th_TT}, we state the following result.

\begin{theorem}\label{th_RR}
Let
\begin{align*}
\left\{\begin{array}{ll}
U\left(p,x\right)&=-2I\sqrt{p}\dn(-\sqrt{p}x,p')\sn(-\sqrt{p}x,p')-2p\cn^2(-\sqrt{p}x,p')+1-2/p,\\
\widetilde{U}\left(p,x\right)&=-2I\sqrt{p}\dn(-\sqrt{p}x,p')\sn(-\sqrt{p}x,p')+2p\cn^2(-\sqrt{p}x,p')-1+2/p,
\end{array}\right.
\end{align*}
where $p'=\sqrt{1-1/p}$. Then
\begin{align*}
\left\{\begin{array}{ll}
RO(x,p,q)&=\frac{\sqrt{p}(1-q)}{2(1+q)}\left(U\left(\frac{1-q^2}{1-p},\sqrt{p-1}x-\ell_{p,q}\right)+\widetilde{U}\left(\frac{1-q^2}{1-p},\sqrt{p-1}x+\ell_{p,q}\right)\right),\\
RE(x,p,q)&=\frac{1-q}{2(1+q)}\left(\widetilde{U}\left(\frac{1-q^2}{1-p},\sqrt{p-1}x+\ell_{p,q}\right)-U\left(\frac{1-q^2}{1-p},\sqrt{p-1}x-\ell_{p,q}\right)\right).
\end{array}\right.
\end{align*}
\end{theorem}
\begin{proof}
By Remark~\ref{GenJP-2},
we obtain
\begin{align}\label{eqrr11}
\left\{\begin{array}{ll}
\frac{\sqrt{p}(1+q)}{1-q}RO(x,p,q)-\frac{1+q}{1-q}RE(x,p,q)&=W\left(\frac{1-q^2}{1-p},\sqrt{p-1}x-\ell_{p,q}\right),\\
\frac{\sqrt{p}(1+q)}{1-q}RO(x,p,q)+\frac{1+q}{1-q}RE(x,p,q)&=\widetilde{W}\left(\frac{1-q^2}{1-p},\sqrt{p-1}x+\ell_{p,q}\right),\\
\end{array}\right.
\end{align}
for some functions $W$ and $\widetilde{W}$. Thus, at $q=0$, we have
\begin{align*}
\left\{\begin{array}{ll}
\sqrt{p}RO(I\sqrt{p}x,1-1/p,0)-RE(I\sqrt{p}x,1-1/p,0)&=W\left(p,x\right),\\
\sqrt{p}RO(I\sqrt{p}x,1-1/p,0)+RE(I\sqrt{p}x,1-1/p,0)&=\widetilde{W}\left(p,x\right),\\
\end{array}\right.
\end{align*}
where $I^2=-1$. Therefore, if we set
$$RE(x,p,0)=2p\cn^2(Ix,\sqrt{p})-2p+1\mbox{ and }RO(x,p,0)=-2I\dn(Ix,\sqrt{p})\sn(Ix,\sqrt{p}),$$
then
\begin{align*}
\left\{\begin{array}{ll}
-2I\sqrt{p}\dn(-\sqrt{p}x,p')\sn(-\sqrt{p}x,p')-2p\cn^2(-\sqrt{p}x,p')+1-2/p&=W\left(p,x\right),\\
-2I\sqrt{p}\dn(-\sqrt{p}x,p')\sn(-\sqrt{p}x,p')+2p\cn^2(-\sqrt{p}x,p')-1+2/p&=\widetilde{W}\left(p,x\right),
\end{array}\right.
\end{align*}
where $p'=\sqrt{1-1/p}$. By \eqref{eqrr11}, we have
\begin{align*}
\left\{\begin{array}{ll}
RO(x,p,q)&=\frac{\sqrt{p}(1-q)}{2(1+q)}\left(W\left(\frac{1-q^2}{1-p},\sqrt{p-1}x-
\ell_{p,q}\right)+\widetilde{W}\left(\frac{1-q^2}{1-p},\sqrt{p-1}x+\ell_{p,q}\right)\right),\\
RE(x,p,q)&=\frac{1-q}{2(1+q)}\left(\widetilde{W}\left(\frac{1-q^2}{1-p},\sqrt{p-1}x+\ell_{p,q}\right)-
W\left(\frac{1-q^2}{1-p},\sqrt{p-1}x-\ell_{p,q}\right)\right),
\end{array}\right.
\end{align*}
which agrees with the case $q=0$. To complete the proof, we have to check that the functions $RO$ and $RE$ satisfy Lemma~\ref{lem:rr1}, which is a routine procedure.
\end{proof}

\section{Applications}\label{Section-4}
\hspace*{\parindent}

In this section, we apply the results obtained in the previous section to present new characterizations for several combinatorial sequences.
\subsection{Peaks, descents and perfect matchings}
\hspace*{\parindent}

Perhaps one of the most important permutation statistics is the peaks statistic (see, e.g., \cite{Ma121,Ma122,Ma1303,Petersen07} and the references contained therein).
A {\it left peak} in $\pi$ is an index $i\in[n-1]$ such that $\pi(i-1)<\pi(i)>\pi(i+1)$, where we take $\pi(0)=0$.
Denote by $\widetilde{P}_{n,k}$ the number of permutations in $\msn$ with $k$ left peaks.
Recall that $P_{n,k}$ is the number of permutations in $\msn$ with $k$ interior peaks.
Define polynomials
$$P_n(x)=\sum_{k=0}^{\lrf{\frac{n-1}{2}}}P_{n,k}x^k,\quad
\widetilde{P}_n(x)=\sum_{k=0}^{\lrf{\frac{n}{2}}}\widetilde{P}_{n,k}x^k.$$
The polynomial $P_n(x)$ satisfies recurrence relation
\begin{equation*}\label{Pnx-1}
P_{n+1}(x)=(nx-x+2)P_n(x)+2x(1-x)\frac{d}{dx}P_n(x),
\end{equation*}
with the initial values
$P_1(x)=1,P_2(x)=2,P_3(x)=4+2x$, and the polynomial $\widetilde{P}_n(x)$ satisfies recurrence relation
\begin{equation}\label{Pnx-2}
\widetilde{P}_{n+1}(x)=(nx+1)\widetilde{P}_{n}(x)+2x(1-x)\frac{d}{dx}\widetilde{P}_{n}(x),
\end{equation}
with the initial values $\widetilde{P}_1(x)=1,\widetilde{P}_2(x)=1+x,\widetilde{P}_3(x)=1+5x$  (see~\cite[A008303,A008971]{Sloane}).

A {\it descent} of a permutation $\pi\in\msn$
is a position $i$ such that $\pi(i)>\pi(i+1)$. Denote by $\des(\pi)$ the number of descents of $\pi$.
 Let
\begin{equation*}
A_n(x)=\sum_{\pi\in\msn}x^{\des(\pi)}=\sum_{k=0}^{n-1}\Eulerian{n}{k}x^{k}.
\end{equation*}
The polynomial $A_n(x)$ is called an {\it Eulerian polynomial}.
Let $B_n$ denote the set of signed permutations of $\pm[n]$ such that $\pi(-i)=-\pi(i)$ for all $i$, where $\pm[n]=\{\pm1,\pm2,\ldots,\pm n\}$.
Let
$${B}_n(x)=\sum_{k=0}^nB(n,k)x^{k}=\sum_{\pi\in B_n}x^{\des_B(\pi)},$$
where
$\des_B(\pi)=|\{i\in[n]:\pi(i-1)>\pi({i})\}|$
with $\pi(0)=0$.
The polynomial $B_n(x)$ is called an {\it Eulerian polynomial of type $B$}, while $B(n,k)$ is called an {\it Eulerian number of type $B$}.


Recall that a {\it perfect matching} of $[2n]$ is a partition of $[2n]$ into $n$ blocks of size $2$.
Denote by $N({n,k})$ the number of perfect matchings of $[2n]$ with the restriction that only $k$ matching pairs
have odd smaller entries (see~\cite[{A185411}]{Sloane}). It is easy to verify that
\begin{equation}\label{recurrence-11}
N({n+1,k})=2kN({n,k})+(2n-2k+3)N({n,k-1}).
\end{equation}

We can now conclude the following result from the discussion above.
\begin{theorem}\label{mainthm:01}
For $n\geq 1$, we have
\begin{enumerate}
\item[\rm(i)] $\sum_{i,j\geq 0}a_{n,i,j}=(2n-1)!!$.
\item[\rm(ii)]$\sum_{j\geq 0}a_{n,i,j}=N(n,n-i)$.
\item[\rm(iii)]$\sum_{j\geq 0}a_{n,i,\lrf{\frac{n}{2}}}x^i=\sum_{j\geq 0}a_{n,i,\lrf{\frac{n}{2}}-i}x^i=\widetilde{P}_n(x)$.
\item[\rm(iv)]$\sum_{j\geq 0}c_{n,i,j}=2^n\Eulerian{n}{i}$.
\item[\rm(v)]$\sum_{j\geq 0}d_{n,i,j}=B(n,i)$.
\item[\rm(vi)] $\sum_{i\geq 0}c_{n,i,\lrf{\frac{n}{2}}}x^i=\sum_{i\geq 0}c_{n,i,\lrf{\frac{n}{2}}-i}x^i=P_{n+1}(x)$.
\item[\rm(vii)]$\sum_{i\geq 0}c_{2n-1,i,0}x^{2n-2-i}=\sum_{i\geq 0}c_{2n,i,0}x^{2n-1-i}=P_{2n}(x)$.
\item[\rm(viii)]$\sum_{i\geq 0}d_{n,i,\lrc{\frac{n}{2}}}x^i=\widetilde{P}_{n}(x)$ and $\sum_{i\geq 0}d_{n,i,\lrc{\frac{n}{2}}-i}x^i=\widetilde{P}_{n+1}(x)$.
\item[\rm(ix)]$\sum_{i\geq 0}d_{2n,i,0}x^{2n-i}=\sum_{i\geq 0}d_{2n+1,i,0}x^{2n+1-i}=\widetilde{P}_{2n+1}(x)$.
\end{enumerate}
\end{theorem}
\begin{proof}
We only prove the assertion for the sequence $a_{n,i,j}$ and the corresponding assertion for the other sequences follows from similar consideration.

(A)\quad Setting $p,q=1$ in Lemma \ref{lem:aa1} gives
\begin{align*}
\left\{\begin{array}{ll}
AO_x(x,1,1)&=AE(x,1,1)+2xAE_x(x,1,1),\\
AE_x(x,1,1)&=AO(x,1,1)+2xAO_x(x,1,1),
\end{array}\right.
\end{align*}
which implies $A_x(x,1,1)=A(x,1,1)+2xA_x(x,1,1)$. Therefore,
$$A(x,1,1)=\frac{A(0,1,1)}{\sqrt{1-2x}}=\frac{1}{\sqrt{1-2x}}=\sum_{n\geq0}\frac{n!}{2^n}\binom{2n}{n}\frac{x^n}{n!}.$$
Hence,
$\sum_{i,j\geq0}a_{n,i,j}=\frac{n!}{2^n}\binom{2n}{n}=(2n-1)!!,$
as required.

(B)\quad Setting $q=1$ in Lemma \ref{lem:aa1} gives
$$A_x(x,p,1)=A(x,p,1)+2p(1-p)A_p(x,p,1)+2xpA_x(x,1,1).$$
By $A(0,1,p)=1$, it is a routine to check that $A(x,p,1)=\frac{\sqrt{1-p}e^{x(1-p)}}{\sqrt{1-pe^{2x(1-p)}}}$. Therefore, by \cite[eq.~(25)]{Ma13} we have
\begin{align*}
A(px,1/p,1)=\frac{\sqrt{1-p}}{\sqrt{1-pe^{2x(1-p)}}}=\sum_{n,k\geq0}N(n,k)x^np^k,
\end{align*}
which implies that $A(x,p,1)=\sum_{n,k\geq0}N(n,n-k)x^np^k$. Hence $\sum_{j\geq0}a_{n,k,j}=N(n,n-k)$, as claimed.

(C)\quad
Let $f_{n,i}=a_{n,i,\lfloor n/2\rfloor}$. By \eqref{a-recu-1} and \eqref{a-recu-2}, we have
$$f_{n,i}=(2i+1)f_{n-1,i}+(n-2i+1)f_{n-1,i-1},\quad 0\leq i\leq \lfloor n/2\rfloor,$$
with $f_{0,0}=1$.
Define $f_n(x)=\sum_{i\geq0}f_{n,i}x^i$. Then
\begin{equation}\label{fny}
f_{n+1}(x)=(nx+1)f_{n}(x)+2x(1-x)\frac{d}{dx}f_{n}(x),
\end{equation}
with the initial condition $f_0(x)=1$.
By comparing~\eqref{fny} with~\eqref{Pnx-2}, we see that the polynomials $f_n(x)$ satisfy the same recurrence relation and
initial conditions as $\widetilde{P}_n(x)$, so they agree. Similarly, it is easy to verify that $$\sum_{j\geq 0}a_{n,i,\lrf{\frac{n}{2}}-i}x^i=\widetilde{P}_n(x),$$
which completes the proof.
\end{proof}

\subsection{Alternating runs and up-down runs}
\hspace*{\parindent}

Let $\pi=\pi(1)\pi(2)\cdots \pi(n)\in\msn$.
We say that $\pi$ changes
direction at position $i$ if either $\pi({i-1})<\pi(i)>\pi(i+1)$, or
$\pi(i-1)>\pi(i)<\pi(i+1)$, where $i\in\{2,3,\ldots,n-1\}$.
We say that $\pi$ has $k$ {\it alternating
runs} if there are $k-1$ indices $i$ such that $\pi$ changes
direction at these positions. The {\it up-down runs} of a permutation $\pi$ are the alternating runs of $\pi$ endowed with a 0
in the front.
Let $R(n,k)$ (resp.~$a_k(n)$) be the number of permutations of $\msn$ with $k$ alternating runs (resp. up-down runs).
For $n,k\ge 1$, the numbers $R(n,k)$ and $a_k(n)$ respectively satisfy the recurrence relations
\begin{equation*}\label{rnk-recurrence01}
R(n,k)=kR(n-1,k)+2R(n-1,k-1)+(n-k)R(n-1,k-2),
\end{equation*}
\begin{equation*}\label{rnk-recurrence02}
a_k(n)=ka_k(n-1)+a_{k-1}(n-1)+(n-k+1)a_{k-2}(n-1),
\end{equation*}
where $R(1,0)=a_0(0)=a_1(1)=1$ and $R(1,k)=a_0(n)=a_k(0)=0$ for $n,k\ge 1$
(see~\cite{Ma1303,Sta08}).




As in the proof of Theorem~\ref{mainthm:01}, it is a routine exercise to show
the following result.
\begin{theorem}\label{mainthm:02}
For $n\geq 1$, we have
\begin{enumerate}
\item[\rm(i)] $\sum_{i,j\geq 0}t_{n,i,j}=\sum_{i,j\geq 0}r_{n-1,i,j}=n!$.
\item[\rm(ii)] $\sum_{i\geq 0}t_{n,i,j}=a_{n-j}(n)$.
\item[\rm(iii)] $\sum_{j\geq 0}t_{n,i,j}=\widetilde{P}_{n,\lrf{n/{2}}-i}$.
\item[\rm(iv)]$\sum_{i\geq 0}r_{n,i,j}=R(n+1,n-j)$.
\item[\rm(v)]$\sum_{j\geq 0}r_{n,i,j}=P_{n+1,\lrf{{n}/{2}}-i}$.
\end{enumerate}
\end{theorem}

%

For convenience, we list the tables of the values of $t_{n,i,j}$ and $r_{n,i,j}$ for $1\leq n\leq 4$.

\quad\quad\quad\quad\quad\quad\quad\quad\begin{tabular}{|p{2.0in}|c|} \hline
\centering \textbf{$t_{1,i,j}$} & $j=0$  \\ \hline
\centering $i=0$ & 1 \\
 \hline
\end{tabular}

\quad\quad\quad\quad\quad\quad\quad\quad\begin{tabular}{|p{2.0in}|c|c|} \hline
\centering \textbf{$t_{2,i,j}$} & $j=0$ &$j=1$ \\ \hline
\centering $i=0$ & 0 & 1 \\
\centering $i=1$ & 1& 0 \\ \hline
\end{tabular}

\quad\quad\quad\quad\quad\quad\quad\quad\begin{tabular}{|p{2.0in}|c|c|c|} \hline
\centering \textbf{$t_{3,i,j}$} & $j=0$ &$j=1$&$j=2$ \\ \hline
\centering $i=0$ & 1 & 3&1 \\
\centering $i=1$ & 1& 0&0 \\ \hline
\end{tabular}

\quad\quad\quad\quad\quad\quad\quad\quad\begin{tabular}{|p{2.0in}|c|c|c|c|} \hline
\centering \textbf{$t_{4,i,j}$} & $j=0$ &$j=1$&$j=2$ &$j=3$\\ \hline
\centering $i=0$ & 0 & 1&3 &1\\
\centering $i=1$ & 4& 10&4 &0\\
\centering $i=2$ & 1& 0&0 &0\\ \hline
\end{tabular}



\quad\quad\quad\quad\quad\quad\quad\quad\begin{tabular}{|p{2.0in}|c|} \hline
\centering \textbf{$r_{1,i,j}$} & $j=0$  \\ \hline
\centering $i=0$ & 2 \\
 \hline
\end{tabular}

\quad\quad\quad\quad\quad\quad\quad\quad\begin{tabular}{|p{2.0in}|c|c|} \hline
\centering \textbf{$r_{2,i,j}$} & $j=0$ &$j=1$ \\ \hline
\centering $i=0$ & 0 & 2 \\
\centering $i=1$ & 4& 0 \\ \hline
\end{tabular}\\

\quad\quad\quad\quad\quad\quad\quad\quad\begin{tabular}{|p{2.0in}|c|c|c|} \hline
\centering \textbf{$r_{3,i,j}$} & $j=0$ &$j=1$ &$j=2$\\ \hline
\centering $i=0$ & 2& 12& 2\\
\centering $i=1$ & 8& 0&0 \\ \hline
\end{tabular}\\

\quad\quad\quad\quad\quad\quad\quad\quad\begin{tabular}{|p{2.0in}|c|c|c|c|} \hline
\centering \textbf{$r_{4,i,j}$} & $j=0$ &$j=1$ &$j=2$&$j=3$\\ \hline
\centering $i=0$ & 0& 2& 12&2\\
\centering $i=1$ & 16& 56&16 &0\\
\centering $i=2$ & 16& 0&0& 0\\ \hline
\end{tabular}\\



Define
\begin{equation*}
\begin{split}
\sn(x,k)&=\sum_{n\geq 0}(-1)^nJ_{2n+1}(k^2)\frac{x^{2n+1}}{(2n+1)!},\\
\cn(x,k)&=1+\sum_{n\geq 0}(-1)^nJ_{2n}(k^2)\frac{x^{2n}}{(2n)!}.
\end{split}
\end{equation*}
Note that $$J_n(k^2)=\sum_{0\leq 2i\leq n-1}J_{n,2i}k^{2i}.$$
Dumont~\cite[Corollary 1]{Dumont79} found that $s_{2n,i,0}=J_{2n,2i}$ and $s_{2n+1,i,0}=J_{2n+2,2i}$.
By comparing~\eqref{Dumont-recu} with~\eqref{t-recu-1} and~\eqref{t-recu-2}, we immediately get
the following result.
\begin{theorem}\label{Jacobi-tnij}
For $n\geq 1$, we have $J_{n,2i}=t_{n,\lrf{{n}/{2}}-i,0}$.
\end{theorem}

It follows from {\it Leibniz's formula} that
\begin{equation*}
\begin{split}
D^{2n+1}(w)&=D^{2n}(wx)\\
&=\sum_{k\geq 0}\binom{2n}{2k}D^{2k}(w)D^{2n-2k}(x)+\sum_{k\geq 0}\binom{2n}{2k+1}D^{2k+1}(w)D^{2n-2k-1}(x),
\end{split}
\end{equation*}
and similarly,
\begin{equation*}
\begin{split}
D^{2n+2}(w)&=D^{2n+1}(wx)\\
&=\sum_{k\geq 0}\binom{2n+1}{2k}D^{2k}(w)D^{2n+1-2k}(x)+\sum_{k\geq 0}\binom{2n+1}{2k+1}D^{2k+1}(w)D^{2n-2k}(x).
\end{split}
\end{equation*}

Therefore, combining~\eqref{Schett-Dumontsystem}, we get
\begin{equation*}
\begin{split}
t_{2n+1,i,0}&=\sum_{k\geq 0}\binom{2n}{2k}\sum_{j=0}^it_{2k,j,0}s_{2n-2k,i-j,0},\\
t_{2n+2,i+1,0}&=\sum_{k\geq 0}\binom{2n+1}{2k+1}\sum_{j=0}^it_{2k+1,j,0}s_{2n-2k,i-j,0}.
\end{split}
\end{equation*}

Thus,
as a corollary of Theorem~\ref{Jacobi-tnij}, we get the following.
\begin{corollary}[{\cite[eq.~(20)]{Viennot80}}]
For $n\geq 0$, we have
\begin{equation*}\label{convolution}
\begin{split}
J_{2n+1,2n-2i}&=\sum_{k\geq 0}\binom{2n}{2k}\sum_{j=0}^iJ_{2k,2k-2j}J_{2n-2k,2i-2j},\\
J_{2n+2,2n-2i}&=\sum_{k\geq 0}\binom{2n+1}{2k+1}\sum_{j=0}^iJ_{2k+1,2k-2j}J_{2n-2k,2i-2j}.
\end{split}
\end{equation*}
\end{corollary}

Let $s_{n,i,j}$ be the numbers defined by~\eqref{Schett-Dumontsystem}.
Set $\widetilde{s}_{n,i,j}=s_{n,j,i}$, i.e., $$\widetilde{s}_{n,i,j}=|\{\pi\in\msn: X(\pi)=j, Y(\pi)=i\}|,$$
where $X(\pi)$ (resp., $Y(\pi)$) is the number of odd (resp., even) cycle peaks of $\pi$.
Based on empirical evidence, we conjecture that
\begin{equation*}
\begin{split}
\widetilde{s}_{2n+1,i,0}&=t_{2n+1,i,0},\\
\widetilde{s}_{2n+1,i,j}&=t_{2n+1,i,2j-1}+t_{2n+1,i,2j}\quad\textrm{for $j\ge 1$},\\
\widetilde{s}_{2n,i,j}&=t_{2n,i,2j}+t_{2n,i,2j+1}\quad\textrm{for $j\ge 0$}.
\end{split}
\end{equation*}

\subsection*{Acknowledgements}
S.-M. Ma is supported by NSFC (11401083), Natural Science Foundation of Hebei Province (A2017501007) and
the Fundamental Research Funds for the Central Universities (N152304006).
This work was finished while Y.-N. Yeh was visiting the School of Mathematical Sciences, Dalian University of Technology, Dalian, P.R. China and he is supported partially by NSC under the Grant No. 104-2115-M-001-010.


\begin{thebibliography}{33}
\bibitem{Abel1826}
N. Abel. Recherches sur les fonctions elliptiques, J. f\"{u}r die reine und angewandte
Mathematik, 1826, 2: 101--181.



\bibitem{Chen93}
W.Y.C. Chen. Context-free grammars, differential operators and formal
power series, Theoret Comput Sci, 1993, 117: 113--129.

\bibitem{Chen17}
W.Y.C. Chen, A.M. Fu. Context-free grammars for permutations and increasing trees, Adv in Appl Math, 2017, 82: 58--82.

%

\bibitem{Dumont79}
D. Dumont. A combinatorial interpretation for the
Schett recurrence on the Jacobian elliptic functions, Math Comp, 1979, 33: 1293--1297.


\bibitem{Dumont81}
D. Dumont. Une approche combinatoire des fonctions elliptiques de Jacobi, Adv Math,
1981, 1: 1--39.

\bibitem{Dumont96}
D. Dumont. Grammaires de William Chen et d\'erivations dans les arbres et
arborescences, S\'em Lothar Combin, Art. B37a, 1996, 37: 1--21.



\bibitem{Flajolet89}
P. Flajolet, Jean Fran\c{c}on. Elliptic functions, continued fractions and doubled permutations,
European J Combin, 1989, 10: 235--241.

\bibitem{Flajolet09}
P. Flajolet, R. Sedgewick. Analytic Combinatorics. Cambridge University Press, 2009.


\bibitem{Jacobi1829}
C. Jacobi. Fundamenta Nova Theoriae Functionum Ellipticarum, 1829. In
Gesammelte Werke (Collected Works), v.~1, 49--239, Berlin, 1881. Reprinted
by Chelsea Press (1965) and available from the American Mathematical Society.

\bibitem{Joyal81}
A. Joyal. Une th\'{e}orie combinatoire des s\'{e}ries formelles, Adv Math, 1981, 42: 1--82.


\bibitem{Ma121}
S.-M. Ma. Derivative polynomials and enumeration of permutations
by number of interior and left peaks, Discrete Math, 2012, 312: 405--412.

\bibitem{Ma122}
S.-M. Ma. An explicit formula for the number of permutations with
a given number of alternating runs, J Combin Theory Ser A, 2012, 119: 1660--1664.

\bibitem{Ma13}
S.-M. Ma. A family of two-variable derivative polynomials for tangent and secant, Electron J Combin, 2013, 20(1): \#P11.


\bibitem{Ma1302}
S.-M. Ma. Some combinatorial arrays generated by context-free grammars, European J Combin, 2013, 34: 1081--1091.


\bibitem{Ma1303}
S.-M. Ma. Enumeration of permutations by number of alternating runs, Discrete Math, 2013, 313: 1816--1822.


\bibitem{Petersen07}
T.K. Petersen. Enriched P-partitions and peak algebras, Adv Math, 2007, 209: 561--610.

\bibitem{Schett76}
A. Schett. Properties of the Taylor series expansion coefficients of the Jacobian
elliptic functions, Math Comp, 1976, 30: 143--147.


\bibitem{Sloane}
N.J.A. Sloane. The On-Line Encyclopedia of Integer Sequences,
published electronically at
http://oeis.org, 2010.



\bibitem{Sta08}
R.P. Stanley. Longest alternating subsequences of permutations, Michigan Math J, 2008, 57: 675--687.


\bibitem{Viennot80}
G. Viennot. Une interpr\'{e}tation combinatoire des coefficients des d\'{e}veloppements
en s\'{e}rie enti\`{e}re des fonctions elliptiques de Jacobi, J Combin Theory Ser A, 1980, 29: 121--133.



\end{thebibliography}
\end{document}